\documentclass[10pt,oneside]{amsart}

\usepackage[
paper=a4paper,
headsep=15pt,text={136mm,207mm},centering,includehead
]{geometry}

\iftrue
\makeatletter
\def\@maketitle{%
  \normalfont\normalsize
  \@adminfootnotes
  \@mkboth{\@nx\shortauthors}{\@nx\shorttitle}%
  \global\topskip42\p@\relax 
  \@settitle
  \ifx\@empty\authors \else \@setauthors \fi
  \ifx\@empty\@dedicatory
  \else
    \baselineskip22\p@
    \vtop{{\small\itshape\@dedicatory\@@par}%
      \global\dimen@i\prevdepth}\prevdepth\dimen@i
  \fi
  \@setabstract
  \normalsize
  \if@titlepage
    \newpage
  \else
    \dimen@25\p@ \advance\dimen@-\baselineskip
    \vskip\dimen@\relax
  \fi
} 
\def\@settitle{%
  \vspace*{-15pt}
  \begin{flushleft}%
    \LARGE\bfseries
    \strut\@title\strut
  \end{flushleft}%
}
\def\@setauthors{%
  \begingroup
  \def\thanks{\protect\thanks@warning}%
  \trivlist
  \raggedright
  \large \@topsep27\p@\relax
  \advance\@topsep by -\baselineskip
  \item\relax
  \author@andify\authors
  \def\\{\protect\linebreak}%
  \authors
  \ifx\@empty\contribs
  \else
    ,\penalty-3 \space \@setcontribs
    \@closetoccontribs
  \fi
  \normalfont
  \endtrivlist
  \endgroup
}
\def\@setaddresses{\par
  \nobreak \begingroup
  \small\raggedright
  \def\author##1{\nobreak\addvspace\smallskipamount}%
  \def\\{\unskip, \ignorespaces}%
  \interlinepenalty\@M
  \def\address##1##2{\begingroup
    \par\addvspace\bigskipamount\noindent
    \@ifnotempty{##1}{(\ignorespaces##1\unskip) }%
    {\ignorespaces##2}\par\endgroup}%
  \def\curraddr##1##2{\begingroup
    \@ifnotempty{##2}{\nobreak\noindent\curraddrname
      \@ifnotempty{##1}{, \ignorespaces##1\unskip}\/:\space
      ##2\par}\endgroup}%
  \def\email##1##2{\begingroup
    \@ifnotempty{##2}{\nobreak\noindent E-mail address%
      \@ifnotempty{##1}{, \ignorespaces##1\unskip}\/:\space
      \ttfamily##2\par}\endgroup}%
  \def\urladdr##1##2{\begingroup
    \def~{\char`\~}%
    \@ifnotempty{##2}{\nobreak\noindent\urladdrname
      \@ifnotempty{##1}{, \ignorespaces##1\unskip}\/:\space
      \ttfamily##2\par}\endgroup}%
  \addresses
  \endgroup
  \global\let\addresses=\@empty
}
\def\@setabstracta{%
    \ifvoid\abstractbox
  \else
    \skip@15pt \advance\skip@-\lastskip
    \advance\skip@-\baselineskip \vskip\skip@
    \box\abstractbox
    \prevdepth\z@ 
    \vskip-15pt
  \fi
}
\renewenvironment{abstract}{%
  \ifx\maketitle\relax
    \ClassWarning{\@classname}{Abstract should precede
      \protect\maketitle\space in AMS document classes; reported}%
  \fi
  \global\setbox\abstractbox=\vtop \bgroup
    \normalfont\small
    \list{}{\labelwidth\z@
      \leftmargin0pc \rightmargin\leftmargin
      \listparindent\normalparindent \itemindent\z@
      \parsep\z@ \@plus\p@
      
    }%
    \item[\hskip\labelsep\bfseries\abstractname.]%
}{%
  \endlist\egroup
  \ifx\@setabstract\relax \@setabstracta \fi
}

\def\ps@headings{\ps@empty
  \def\@evenhead{%
    \setTrue{runhead}%
    \normalfont\scriptsize
    \rlap{\thepage}\hfill
    \def\thanks{\protect\thanks@warning}%
    \leftmark{}{}}%
  \def\@oddhead{%
    \setTrue{runhead}%
    \normalfont\scriptsize
    \def\thanks{\protect\thanks@warning}%
    \rightmark{}{}\hfill \llap{\thepage}}%
  \let\@mkboth\markboth
}\ps@headings

\def\section{\@startsection{section}{1}%
  \z@{-1.4\linespacing\@plus-.5\linespacing}{.8\linespacing}%
  {\normalfont\bfseries\Large}}
\def\subsection{\@startsection{subsection}{2}%
  \z@{-.8\linespacing\@plus-.3\linespacing}{.5\linespacing\@plus.2\linespacing}%
  {\normalfont\bfseries\large}}
\def\subsubsection{\@startsection{subsubsection}{3}%
  \z@{.7\linespacing\@plus.2\linespacing}{-1.5ex}%
  {\normalfont\bfseries}}
\def\@secnumfont{\bfseries}

\renewcommand\contentsnamefont{\bfseries}
\def\@starttoc#1#2{\begingroup
  \setTrue{#1}%
  \par\removelastskip\vskip\z@skip
  \@startsection{}\@M\z@{\linespacing\@plus\linespacing}%
    {.5\linespacing}{
      \contentsnamefont}{#2}%
  \ifx\contentsname#2%
  \else \addcontentsline{toc}{section}{#2}\fi
  \makeatletter
  \@input{\jobname.#1}%
  \if@filesw
    \@xp\newwrite\csname tf@#1\endcsname
    \immediate\@xp\openout\csname tf@#1\endcsname \jobname.#1\relax
  \fi
  \global\@nobreakfalse \endgroup
  \addvspace{32\p@\@plus14\p@}%
  \let\tableofcontents\relax
}
\def\contentsname{Contents}
\def\l@section{\@tocline{2}{.5ex}{0mm}{5pc}{}}
\def\l@subsection{\@tocline{2}{0pt}{2em}{5pc}{}}
\makeatother
\fi 

\usepackage{amssymb}
\usepackage[all]{xy}
\usepackage{graphicx,color,float}
\usepackage[bookmarks]{hyperref}
\usepackage{pinlabel}
\usepackage[textwidth=1.3in,color=yellow]{todonotes}
\usepackage{amsmath}
\usepackage{enumitem,calc}
\usepackage{pb-diagram,pb-xy}
\usepackage{tikz}
\usetikzlibrary{calc,decorations.pathreplacing}

\def\Z{\mathbb{Z}}
\def\Q{\mathbb{Q}}

\def\C{\mathcal{C}}

\def\d{\partial}

\def\+{\oplus}

\newtheoremstyle{theorem-giventitle}
        {}{}              
        {\itshape}                      
        {}                              
        {\bfseries}                     
        {.}                             
        {\thm@headsep}                             
        {\thmnote{\bfseries#3}}
\newtheoremstyle{theorem-givenlabel}
        {}{}              
        {\itshape}                      
        {}                              
        {\bfseries}                     
        {.}                             
        {\thm@headsep}                             
        {\thmname{#1}~\thmnumber{#3}\setcurrentlabel{#3}}

\newtheoremstyle{definition-giventitle}
        {}{}              
        {}                      
        {}                              
        {\bfseries}                     
        {.}                             
        {\thm@headsep}                             
        {\thmnote{\bfseries#3}}
\def\setcurrentlabel#1{\gdef\@currentlabel{#1}}

\makeatother

\newtheorem{theorem}{Theorem}[section]
\newtheorem{theoremalpha}{Theorem}

\newtheorem{lemma}[theorem]{Lemma}

\theoremstyle{definition}

\newtheorem*{question}{Question}

\theoremstyle{theorem-giventitle}
\newtheorem{theorem-named}{}
\theoremstyle{theorem-givenlabel}
\newtheorem{theorem-labeled}{Theorem}

\theoremstyle{definition-giventitle}
\newtheorem{definition-named}{}
\newtheorem{step-named}{}

\numberwithin{equation}{section}

\def\to{\mathchoice{\longrightarrow}{\rightarrow}{\rightarrow}{\rightarrow}}
\makeatletter
\newcommand{\shortxra}[2][]{\ext@arrow 0359\rightarrowfill@{#1}{#2}}
\def\longrightarrowfill@{\arrowfill@\relbar\relbar\longrightarrow}
\newcommand{\longxra}[2][]{\ext@arrow 0359\longrightarrowfill@{#1}{#2}}

\makeatother

\begin{document}

\title{Rasmussen $s$-invariants of satellites do not detect slice knots}
\author{Jae Choon Cha}
\address{
  Department of Mathematics\\
  POSTECH\\
  Pohang 790--784\\
  Republic of Korea\quad -- and --
  \linebreak
  School of Mathematics\\
  Korea Institute for Advanced Study \\
  Seoul 130--722\\
  Republic of Korea
}
\email{jccha@postech.ac.kr}

\author{Min Hoon Kim}
\address{
  School of Mathematics\\
  Korea Institute for Advanced Study \\
  Seoul 130--722\\
  Republic of Korea
}
\email{kminhoon@kias.re.kr}

\dedicatory{Dedicated to the memory of Tim Cochran}

\begin{abstract}
  We present a large family of knots for which the Rasmussen
  $s$-invariants of arbitrary satellites do not detect sliceness.
  This answers a question of Hedden.  The proof hinges on work of
  Kronheimer-Mrowka and Cochran-Harvey-Horn.
\end{abstract}

\maketitle

\section{Introduction}

In \cite{Rasmussen:2004-1}, Rasmussen introduced a smooth knot
concordance invariant $s(K)$ using a deformed version of Khovanov
homology.  In general, while invariants from Khovanov homology have
common aspects with and are related to those from Heegaard Floer
homology, it is expected and often confirmed that they behave very
distinctly.  For instance, the volume conjecture tells us that Jones
polynomials of cables contain significantly more information than the
Alexander polynomials of cables (which are completely determined by
the Alexander polynomial of the initially given knot).  In case of the
$s$-invariant, it shares several properties with the $\tau$-invariant
of Ozsv\'ath-Szab\'o and Rasmussen, which may be viewed as its
Heegaard Floer ``analog'', but Hedden and Ording showed that $s$ is
independent from~$\tau$~\cite{Hedden-Ording:2008-1}.  Regarding the
cabling, Hedden asked the following question:

\begin{question}[{\cite[Question 1.4]{Hedden:2009-1}}]
  Does the Rasmussen $s$-invariant, applied to all iterated cables of
  $K$, determine whether $K$ is smoothly slice?
\end{question}

We remark that the behavior of $\tau$ under cabling is well understood
by work of Hedden and
Hom~\cite{Hedden:2005-1,Hedden:2009-1,Hom:2014-2}.  See also
\cite{VanCott:2010-1, Petkova:2013-1}.  In particular the $\tau$
version of the above question was answered in the negative, in a
similar fashion to the Alexander polynomial case (but in a more
sophisticated way)~\cite{Hedden:2009-1,Hom:2014-2}.  The $s$-invariant
case was left open, mainly because of the difficulty of analyzing the
Khovanov chain complex of cables.

\medskip

The goal of this note is to answer Hedden's question on the
$s$-invariant by presenting a large family of counterexamples.  To
state it, we use the following condition for a knot $K$ in~$S^3$,
motivated by work of Kronheimer and
Mrowka~\cite{Kronheimer-Mrowka:2013-1} and Cochran, Harvey, and
Horn~\cite{Cochran-Harvey-Horn:2012-1}:

\begin{enumerate}
  [labelindent=0mm,leftmargin=\widthof{\textbf{(KM)}\space\space},label=*,topsep=\medskipamount]
\item[\textbf{(KM)}] There exist pairs $(V_+,D_+)$, $(V_-,D_-)$ of a compact
  smooth 4-manifold $V_\pm$ and a smoothly embedded disk $D_\pm$ in
  $V_\pm$ such that $\partial(V_\pm,D_\pm)=(S^3,K)$, $b_1(V_\pm)=0$,
  $V_{\pm}$ is $\pm$-definite, i.e., $b^\pm_2(V_\pm)=b_2(V_\pm)$, and
  $[D_\pm, \d D_\pm]=0$ in $\pi_2(V_\pm, S^3)$.
\end{enumerate}

By \cite[Corollary~1.1]{Kronheimer-Mrowka:2013-1}, $s(K)=0$ if $K$
satisfies~(KM)\@.

For knots $K\subset S^3$ and $P \subset S^1\times D^2$, denote by
$P(K)$ the satellite knot with pattern $P$ and companion~$K$.  Denote
the unknot by~$U$.

\begin{theoremalpha}
  \label{theorem:satellite-s-invariant}
  If $K$ is a knot satisfying \textup{(KM)}, then $s(P(K))=s(P(U))$
  for any pattern $P\subset S^1\times D^2$.
\end{theoremalpha}

Consequently, if $K$ satisfies (KM), any iterated cable of $K$ has the
same $s$-invariant as the corresponding iterated cable of the unknot.
This answers Hedden's question in the negative.

The collection of knots satisfying (KM) is large.  For instance,
0-bipolar knots in the sense of \cite{Cochran-Harvey-Horn:2012-1}
satisfy~(KM)\@.  Especially if a knot has a diagram from which a slice
knot is obtained by changing positive crossings and has a (possibly
different) diagram from which a slice knot is obtained by changing
negative crossings, then the knot satisfies
(KM)~\cite[Lemma~3.4]{Cochran-Lickorish:1986-1}.

To describe explicit examples, let $K(a,-b)$ be the knot shown in
Figure~\ref{figure:example}.  Then $K(1,-n)$ and $K(n,-n)$ satisfy
(KM) for any $n>0$ by the above.  It is known that $K(1,-n)$ for $2\ne
n>0$ and $K(n,-n)$ for $n>0$ are not slice, even topologically (e.g.,
see \cite{Casson-Gordon:1986-1, Jiang:1981-1, Kim:2005-2,
  Cha:2003-1}).  In fact, they generate a subgroup isomorphic to
$\Z^\infty \oplus \Z_2^\infty$ in the smooth and topological knot
concordance group.

\begin{figure}[H]

  \begin{tikzpicture}[x=1bp,y=1bp,line width=1pt,join=round]
    \small

    \begin{scope}[double distance=1pt]
      \draw [draw=white,double=black,line width=1.5pt]
      (10+100*0,20) node (last) {} -- ++(0,15) ..controls +(0,12) and +(-12,0).. ++(25,25) -- ++(10,0) 
      (last.center) ++(10,0) -- ++(0,15) ..controls +(0,7) and +(-7,0).. ++(15,15) -- ++(10,0) node (last) {};
      \foreach \j in {0,1} {
        \draw [draw=white,double=black,line width=1.5pt]
        (last.center) ..controls +(5,0) and +(-5,0).. ++(10,+10);
        \draw [draw=white,double=black,line width=1.5pt]
        (last.center) ++(0,10) ..controls +(5,0) and +(-5,0).. ++(10,-10) node (last) {};
      }
      \foreach \j in {0,1,2} {
        \draw [fill=black, line width=0] (last.center) +(3+14*\j/2,5) circle [radius=1];
      }
      \path (last.center) ++(20,0) node (last) {};
      \foreach \j in {0,1} {
        \draw [draw=white,double=black,line width=1.5pt]
        (last.center) ..controls +(5,0) and +(-5,0).. ++(10,+10);
        \draw [draw=white,double=black,line width=1.5pt]
        (last.center) ++(0,10) ..controls +(5,0) and +(-5,0).. ++(10,-10) node (last) {};
      }
      \draw (last.center) -- ++(10,0) ..controls +(7,0) and +(0,7).. ++(15,-15) -- ++(0,-15)
      (last.center) ++(0,10) -- ++(10,0) ..controls +(12,0) and +(0,12).. ++(25,-25) -- ++(0,-15);

      \draw [draw=white,double=black,line width=1.5pt]
      (10+100*1,20) node (last) {} -- ++(0,15) ..controls +(0,12) and +(-12,0).. ++(25,25) -- ++(10,0) 
      (last.center) ++(10,0) -- ++(0,15) ..controls +(0,7) and +(-7,0).. ++(15,15) -- ++(10,0) node (last) {};
      \foreach \j in {0,1} {
        \draw [draw=white,double=black,line width=1.5pt]
        (last.center) ++(0,10) ..controls +(5,0) and +(-5,0).. ++(10,-10);
        \draw [draw=white,double=black,line width=1.5pt]
        (last.center) ..controls +(5,0) and +(-5,0).. ++(10,+10) ++(0,-10) node (last) {};
      }
      \foreach \j in {0,1,2} {
        \draw [fill=black, line width=0] (last.center) +(3+14*\j/2,5) circle [radius=1];
      }
      \path (last.center) ++(20,0) node (last) {};
      \foreach \j in {0,1} {
        \draw [draw=white,double=black,line width=1.5pt]
        (last.center) ++(0,10) ..controls +(5,0) and +(-5,0).. ++(10,-10);
        \draw [draw=white,double=black,line width=1.5pt]
        (last.center) ..controls +(5,0) and +(-5,0).. ++(10,+10) ++(0,-10) node (last) {};
      }
      \draw (last.center) -- ++(10,0) ..controls +(7,0) and +(0,7).. ++(15,-15) -- ++(0,-15)
      (last.center) ++(0,10) -- ++(10,0) ..controls +(12,0) and +(0,12).. ++(25,-25) -- ++(0,-15);

  \end{scope}

  \draw [line width=.5pt, decoration={brace, amplitude=5, raise=5}, decorate] (45,60) -- (105,60)
  node [pos=0.5,anchor=south,yshift=10] {$a$ positive full twists};
  \draw [line width=.5pt, decoration={brace, amplitude=5, raise=5}, decorate] (145,60) -- (205,60)
  node [pos=0.5,anchor=south,yshift=10] {$b$ negative full twists};

  \draw [cap=round]
  (0,0) ..controls +(-10,0) and +(-10,0).. ++ (0,20) -- ++(10,0) 
  ++(10,0) -- ++(90,0)
  ++(10,0) -- ++(10,0)
  ++(10,0) -- ++(90,0)
  ++(10,0) -- ++(10,0) ..controls +(10,0) and +(10,0).. ++(0,-20) -- (0,0);
  
  \end{tikzpicture}

  \caption{The knot $K(a,-b)$.}
  \label{figure:example}
\end{figure}
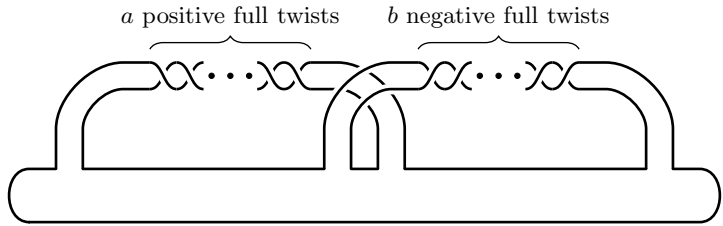

The figure eight knot is the simplest case ($n=1$).  We remark that
$K(n,-n)$ with $n>0$ has order 2 in the knot concordance group, and is
smoothly rationally slice, that is, bounds a smoothly embedded disk in
a rational homology ball whose boundary is~$S^3$.
See~\cite[Theorem~4.16 and Figure~6]{Cha:2003-1} for a proof.
Consequently $P(K(n,-n))$ is rationally concordant to $P(U)$ for
any~$P$.  This relates the case of $K(n,-n)$ to an intriguing open
question (e.g., see~\cite[Question~2.1]{Cha-Powell:2012-1}): if $K$ is
rationally slice, does $s(K)$ vanish?  If so, then one would conclude
immediately that $s(P(K(n,-n)))=s(P(U))$.

Our proof of Theorem~\ref{theorem:satellite-s-invariant}, which is
given in Section~\ref{section:s-invariant-satellite}, shows the
conclusion without the invariance of $s$ under rational concordance.
Our method is largely influenced by work of Cochran, Harvey, and
Horn~\cite{Cochran-Harvey-Horn:2012-1}.

We remark that our argument for the $s$-invariant applies to the case
of the $\tau$-invariant, the $\epsilon$-invariant and the knot Floer
chain complex invariant $[CFK^\infty(K)]$ of
Hom~\cite{Hom:2014-2,Hom:2011-1}, and the $\delta_{p^k}$-invariant of
Manolescu-Owens~\cite{Manolescu-Owens:2007-1} and
Jabuka~\cite{Jabuka:2012-1} as well.  Using this, we observe that
these invariants of arbitrary satellites, even when considered all
together, do not detect slice knots:

\begin{theoremalpha}
  \label{theorem:satellites-of-various-invariants}
  There are knots $K$ which generate a subgroup isomorphic to
  $\Z^\infty \oplus \Z_2^\infty$ in the smooth and topological knot
  concordance groups and satisfy $\bullet(P(K))=\bullet(P(U))$ for any
  pattern $P$ and $\bullet = s$, $\tau$, $\epsilon$, $\delta_{p^k}$,
  $[CFK^\infty(-)]$.
\end{theoremalpha}

The proof of Theorem~\ref{theorem:satellites-of-various-invariants} is
given in Section~\ref{section:invariant-from-floer-homology}.

\subsection*{Acknowledgment} 

The first author was supported by NRF grants 2013067043 and
2013053914.

\section{Proof of Theorem A: $s$-invariant of satellites}
\label{section:s-invariant-satellite}

Henceforce we assume that everything is smooth.  We say that a pattern
$Q\subset S^1\times D^2$ is \emph{slice} if $Q$ viewed as a knot in
$S^1\times D^2\cup D^2\times S^1 =\partial(D^2\times D^2)=S^3$ is
slice, that is, $Q(U)$ is slice.  Let $\C$ be the knot concordance
group.  The following innocent observation reduces the investigation
of the effect of satellite operations on a homomorphism of $\C$ to the
case of slice patterns.  We state it as a lemma for later use as well:

\begin{lemma}
  \label{lemma:main}
  Suppose $f\colon \C\to A$ is an abelian group homomorphism and $K$
  is a knot in $S^3$ satisfying $f(Q(K))=0$ for any slice pattern~$Q$.
  Then $f(P(K))=f(P(U))$ for any pattern~$P$.
\end{lemma}

\begin{proof}
  For a given pattern $P$, let $Q_P$ be the pattern $P\# -P(U)$, that
  is, $(S^1\times D^2, Q_P)$ is the connected sum of two pairs
  $(S^1\times D^2, P)$ and $(S^3, -P(U))$.  Then $Q_P(J)=P(J)\#-P(U)$
  for any knot~$J$.  So $Q_P(U)=P(U)\#-P(U)$ is a ribbon knot in~$S^3$
  and thus $Q_P$ is a slice pattern.  For the given $K$, by the
  hypotheses, we have
  \[
  0 = f(Q_P(K))=f(P(K)\#-P(U))=f(P(K))-f(P(U)). \qedhere
  \]
\end{proof}

The following is a variation of arguments appeared in
\cite[Proposition~3.3]{Cochran-Harvey-Horn:2012-1},
\cite[Theorem~2.6~(6)]{Cha-Powell:2012-1}.

\begin{lemma}
  \label{lemma:satellite-kmcondition}
  Suppose that a knot $K$ in $S^3$ satisfies \textup{(KM)} and
  $Q\subset S^1\times D^2$ is a slice pattern.  Then the satellite
  knot $Q(K)$ satisfies \textup{(KM)}\@.
\end{lemma}

\begin{proof}
  Suppose $K$ satisfies (KM) via $(V_\pm, D_\pm)$.  Choose a slice
  disk $\Delta\subset D^2\times D^2$ bounded by $Q\subset S^1\times
  D^2\subset \d (D^2\times D^2)$.  Choose a diffeomorphism $D^2\times
  D^2\to \nu(D_\pm)$ onto the normal bundle $\nu(D_\pm)$ which sends
  $D^2\times 0$ to $D_\pm$ and $p\times D^2$ to a normal disk fiber
  for each $p\in D^2$.  Let $\Delta_\pm$ be the image of the disk
  $\Delta$ under the diffeomorphism.  Then $\partial\Delta_\pm =
  Q(K)$, that is, $\Delta_\pm$ is a slice disk for $Q(K)$ in $V_\pm$.

  Now, to show that $Q(K)$ satisfies (KM), it suffices to prove that
  $[\Delta_\pm,\d \Delta_\pm]=0$ in $\pi_2(V_\pm, S^3)$.  Consider the
  following commutative diagram of inclusions, where
  $\nu(\partial D_\pm)=\nu(D_\pm)\cap S^3$.
  \[
  \xymatrix{
    (\Delta_\pm,\partial\Delta_\pm) \ar[r] & 
    (\nu(D_\pm),\nu(\partial D_\pm)) \ar[r]^-k &
    (V_\pm, S^3)]
    \\
    & (D_\pm,\partial D_\pm) \ar[u]^-i \ar[ur]_-j
  }
  \]
  The induced map $j_*$ on $\pi_2$ is zero by the condition~(KM)\@.
  Since $i$ is a homotopy equivalence, $k_* = 0$ on $\pi_2$ too.  It
  follows that $[\Delta_\pm,\d \Delta_\pm]=0$ in $\pi_2(V_\pm, S^3)$.
\end{proof}

Now, we are ready to prove
Theorem~\ref{theorem:satellite-s-invariant}:

\begin{proof}[Proof of Theorem~\ref{theorem:satellite-s-invariant}]
  Suppose that $K$ satisfies (KM)\@.  For any slice pattern $Q\subset
  S^1\times D^2$, $Q(K)$ satisfies (KM) by
  Lemma~\ref{lemma:satellite-kmcondition}, and consequently
  $s(Q(K))=0$ by~\cite[Corollary~1.1]{Kronheimer-Mrowka:2013-1}.
  Since $s\colon\C\to\Z$ is a homomorphism, $s(P(K))=s(P(U))$ for all
  pattern $P\subset S^1\times D^2$ by Lemma~\ref{lemma:main}.
\end{proof}

\section{Proof of
  Theorem~\ref{theorem:satellites-of-various-invariants}: invariants
  from Floer homology}
\label{section:invariant-from-floer-homology}

Our argument used in Section~\ref{section:s-invariant-satellite} can be
applied to $\tau$, $\epsilon$, $[CFK^\infty(-)]$, and $\delta_{p^k}$
in a similar way.  We begin with an observation based on Hom's work,
which will be used to reduce the cases of $\epsilon$ and
$[CFK^\infty(-)]$ to the case of~$\tau$.

\begin{theorem}[Hom]
  \label{theorem:hom-equivalence}
  For two knots $K$ and $K'$, the following are equivalent:
  \begin{enumerate}[label=\textup{(\arabic{*})}]
  \item $\epsilon(K\#-K')=0$.
  \item $[CFK^\infty(K)]=[CFK^\infty(K')]$.
  \item $\tau(P(K))=\tau(P(K'))$ for any pattern $P$.
  \item $\epsilon(P(K))=\epsilon(P(K'))$ for any pattern $P$.
  \item $\epsilon(P(K)\#-P(K'))=0$ for any pattern $P$.
  \item $[CFK^\infty(P(K))]=[CFK^\infty(P(K'))]$ for any pattern $P$.
  \end{enumerate}
\end{theorem}

\begin{proof} 
  The equivalences (1)~$\Leftrightarrow$~(2) and (5)~$\Leftrightarrow$~(6)
  are definitions in \cite[Corollary 5]{Hom:2011-1}, and
  (2)~$\Leftrightarrow$~(3) is due to~\cite[Corollary~5]{Hom:2011-1}.
   
  For (3)~$\Rightarrow$~(4), let $P_{\pm}$ be the pattern obtained by
  taking $(2,\pm1)$ cable of $P\subset S^1\times D^2$.  Denote the
  $(p,q)$-cable of a knot $J$ by~$J_{p,q}$.  Then $P(K)_{2,\pm1} =
  P_\pm(K)$.  Hom showed that the value of $\epsilon(J)$ is determined
  by the pair of integers $(\tau(J_{2,1}),
  \tau(J_{2,-1}))$~\cite[Theorem~5.2]{Hom:2011-1}.  Since
  $\tau(P_\pm(K))=\tau(P_\pm(K'))$ by (3), it follows that
  $\epsilon(P(K))=\epsilon(P(K'))$.

  For (4)~$\Rightarrow$~(3), define $Q_P:=P\#-P(K')$ similarly to the
  proof of Theorem~\ref{theorem:satellite-s-invariant}.  Since
  $Q_P(K')=P(K')\#-P(K')$ is slice, $\epsilon(Q_P(K'))=0$.  By (4),
  $\epsilon(Q_P(K))=\epsilon(Q_P(K'))=0$.  Hom showed that $\tau(J)=0$
  whenever $\epsilon(J)=0$~\cite[Theorem~5.2]{Hom:2011-1}.  Applying
  this, it follows that $\tau(Q_P(K))=0$.  Now
  $\tau(P(K))-\tau(P(K'))=\tau(P(K)\#-P(K'))=\tau(Q_P(K))=0$.

  In the above paragraph we have shown that (4) implies
  $\epsilon(Q_P(K))=0$.  Since $Q_P(K) = P(K)\#-P(K')$, this shows
  (4)~$\Rightarrow$~(5).  The converse is a straightforward
  consequence of Hom's result that $\epsilon(J\#J')=\epsilon(J')$
  whenever $\epsilon(J)=0$~\cite[Proposition~3.6~(6)]{Hom:2014-2}:
  indeed, from this and the concordance invariance of $\epsilon$, it follows
  that $\epsilon(J')=\epsilon(J\#-J'\#J')=\epsilon(J)$ if
  $\epsilon(J\#-J')=0$.  The implication (5)~$\Rightarrow$~(4) is a
  special case of this.
\end{proof}

For the cases of $\tau$, $\epsilon$, $[CFK^\infty]$, and
$\delta_{p^k}$, it turns out to be natural to consider the class of
\emph{$R$-homology $n$-bipolar knots} defined
in~\cite{Cha-Powell:2012-1}.  Here $R$ is a subring of~$\Q$.  This is
a homology version of the notion of $n$-bipolar knots introduced
in~\cite{Cochran-Harvey-Horn:2012-1}.  We do not spell out the
definition since we do not use it directly; the readers are referred
to \cite[Definition 2.3]{Cha-Powell:2012-1} for details.  We will use
the following facts only.  For a prime $p$, denote by
$\Z_{(p)}=\{a/b\in \Q \mid a,b\in \Z,\; p\nmid b\}$, the localization
of $\Z$ at~$p$.

\begin{enumerate}[label=(B\arabic{*}),topsep=\medskipamount]
\item\label{item:BP-HBP} An $n$-bipolar knot is $R$-homology
  $n$-bipolar for any~$R$ \cite[p.~1544]{Cha-Powell:2012-1}.
\item\label{item:satellite-HBP} If $K$ is a $R$-homology $n$-bipolar
  knot, then $Q(K)$ is $R$-homology $n$-bipolar for any slice
  pattern~$Q$ \cite[Theorem~2.6~(6)]{Cha-Powell:2012-1}.
\item\label{item:HBP-tau} $\tau(K)=0$ if $K$ is $\Z_{(p)}$-homology
  $0$-bipolar for some prime $p$, or equivalently $\Q$-homology
  bipolar~\cite[Theorem~2.7]{Cha-Powell:2012-1}.
\item\label{item:HBP-delta} $\delta_{p^k}(K)=0$ if $K$ is
  $\Z_{(p)}$-homology
  $1$-bipolar~\cite[Theorem~2.8]{Cha-Powell:2012-1}.
\end{enumerate}

We remark that \ref{item:satellite-HBP}, \ref{item:HBP-tau},
and~\ref{item:HBP-delta} above are mild generalizations of
\cite[Propositions~3.3, 1.2, and~2.8]{Cochran-Harvey-Horn:2012-1}.






\begin{theorem}
  \label{theorem:HBP-satellite-invariants}
  Suppose $K$ is a knot in $S^3$ and $p$ is a prime.
  \begin{enumerate}[label=\textup{(\arabic{*})}]
  \item If $K$ is $\Z_{(p)}$-homology $0$-bipolar,
    $\bullet(P(K))=\bullet(P(U))$ for any pattern $P$ and
    $\bullet=\tau$, $\epsilon$, $[CFK^\infty(-)]$.
  \item If $K$ is $\Z_{(p)}$-homology $1$-bipolar,
    $\delta_{p^k}(P(K))=\delta_{p^k}(P(U))$ for any pattern $P$.
  \end{enumerate}
  Consequently, if $K$ is $1$-bipolar, $\bullet(P(K))=\bullet(P(U))$
  for any pattern $P$ and $\bullet=s$, $\tau$, $\epsilon$,
  $[CFK^\infty(-)]$.
 \end{theorem}

 \begin{proof}
   This is shown by a variation of the proof of
   Theorem~\ref{theorem:satellite-s-invariant}\@, using
   \ref{item:satellite-HBP} in place of Lemma~\ref{lemma:main}.  Also,
   we need to use \ref{item:HBP-tau} and \ref{item:HBP-delta} instead
   of \cite[Corollary~1.1]{Kronheimer-Mrowka:2013-1}.  Details are as
   follows.

   (1) By Theorem \ref{theorem:hom-equivalence}, it suffices to prove
   the $\tau$ case.  Suppose $K$ is $\Z_{(p)}$-homology 0-bipolar.
   For an arbitrary slice pattern $Q$, $Q(K)$ is $\Z_{(p)}$-homology
   0-bipolar by~\ref{item:satellite-HBP}.  It follows that
   $\tau(Q(K))=0$ by~\ref{item:HBP-tau}.  By Lemma~\ref{lemma:main},
   $\tau(P(K))=\tau(P(U))$ for any pattern~$P$.

   (2) Suppose that $K$ is $\Z_{(p)}$-homology 1-bipolar.  For any
   slice pattern $Q$, $Q(K)$ is $\Z_{(p)}$-homology 1-bipolar
   by~\ref{item:satellite-HBP}.  It follows that
   $\delta_{p^k}(Q(K))=0$ by~\ref{item:HBP-delta}.  By
   Lemma~\ref{lemma:main}, $\delta_{p^k}(P(K))=\delta_{p^k}(P(U))$ for
   any pattern~$P$.
\end{proof}

\begin{proof}[Proof of
  Theorem~\ref{theorem:satellites-of-various-invariants}]
  There exists a family of 1-bipolar knots, say $\{K_i\}$, which
  generates a subgroup isomorphic to $\Z^\infty \oplus \Z_2^\infty$ in
  the smooth and topological knot concordance groups~\cite[Theorem
  7.1]{Cochran-Harvey-Horn:2012-1}.  Since a 1-bipolar knot satisfies
  (KM) and is $\Z_{(p)}$-homology 1-bipolar, by
  Theorem~\ref{theorem:HBP-satellite-invariants} we have
  $\bullet(P(K))=\bullet(P(U))$ for any pattern $P$ and
  $\bullet=\tau$, $\epsilon$, $[CFK^\infty(-)]$, $\delta_{p^k}$
  when $K$ is $1$-bipolar, especially when $K=K_i$.
\end{proof}

\bibliographystyle{amsalpha}
\renewcommand{\MR}[1]{}
\bibliography{research}

\end{document}